\documentclass[11pt,english]{article}
\usepackage[T1]{fontenc}
\usepackage[latin9]{inputenc}
\usepackage{amssymb}
\usepackage{amsmath}
\usepackage{amssymb}
\usepackage{amsthm}
\usepackage{amsfonts}
\usepackage{babel}
\usepackage{fullpage}
\usepackage{color}
\usepackage{times}
\usepackage{boites}

\parskip=\medskipamount
\numberwithin{equation}{section}

\newtheorem{theorem}{Theorem}[section]
\newtheorem{lem}[theorem]{Lemma}

\newtheorem{prop}[theorem]{Proposition}

\theoremstyle{definition}
\newtheorem{defn}{Definition}[section]

\theoremstyle{remark}

\newcommand{\noun}[1]{\textsc{#1}}
\newcommand{\Exp}{\operatornamewithlimits{\mathbb{E}}}
\newcommand{\Acc}{\ensuremath{\operatorname{Acc}} }
\newcommand{\Odd}{\ensuremath{\operatorname{Odd}} }
\newcommand{\eps}{\epsilon}
\newcommand{\cube}{\operatorname{\{0, 1\}}}
\newcommand{\pcube}{\operatorname{\{-1,1\}}}
\newcommand{\infl}{\operatorname{I} }
\newcommand{\NP}{\operatorname{NP}}
\newcommand{\PCP}{\operatorname{PCP}}

\usepackage{babel}

\begin{document}

\title{A Hypergraph Dictatorship Test with Perfect Completeness}

\author{Victor Chen %
\thanks{MIT CSAIL. \texttt{victor@csail.mit.edu}. Research supported in part
by an NSF graduate fellowship and NSF Award CCR-0514915.%
}}

\date{}

\maketitle
\setcounter{page}{0} 
\begin{abstract}
A hypergraph dictatorship test is first introduced by Samorodnitsky
and Trevisan in~\cite{SamTre06} and serves as a key component in
their unique games based $\PCP$ construction. Such a test has oracle
access to a collection of functions and determines whether all the
functions are the same dictatorship, or all their low degree influences
are $o(1).$ The test in~\cite{SamTre06} makes $q\geq3$ queries
and has amortized query complexity $1+O\left(\frac{\log q}{q}\right)$
but has an inherent loss of perfect completeness. In this paper we
give an adaptive hypergraph dictatorship test that achieves both
perfect completeness and amortized query complexity $1+O\left(\frac{\log q}{q}\right)$. 
\end{abstract}
\pagenumbering{arabic}

\section{Introduction}

Linearity and dictatorship testing have been studied in the past decade
both for their combinatorial interest and connection to complexity
theory. These tests distinguish functions which are linear/dictator
from those which are far from being a linear/dictator function. The
tests do so by making queries to a function at certain points and
receiving the function's values at these points. The parameters of
interest are the number of queries a test makes and the completeness
and soundness of a test.

In this paper we shall work with boolean functions of the form $f:\cube^{n}\rightarrow\pcube$.
We say a function $f$ is \emph{linear} if $f=(-1)^{\sum_{i\in S}x_{i}}$
for some subset $S\subseteq[n]$. A \emph{dictator} function is simply
a linear function where $|S|=1$, i.e., $f(x)=(-1)^{x_{i}}$ for some
$i$. A dictator function is often called a \emph{long code}, and
it is first used in~\cite{BelGolSud98} for the constructions of
probabilistic checkable proofs ($\PCP$s), see e.g.,~\cite{AroSaf,ALMSS}.
Since then, it has become standard to design a $\PCP$ system as the
composition of two verifiers, an outer verifier and an inner verifier.
In such case, a $\PCP$ system expects the proof to be written in
such a way so that the outer verifier, typically based on the verifier
obtained from Raz's Parallel Repetition Theorem~\cite{Raz98}, selects
some tables of the proof according to some distribution and then passes
the control to the inner verifier. The inner verifier, with oracle
access to these tables, makes queries into these tables and ensures
that the tables are the encoding of some error-correcting codes and
satisfy some joint constraint. The long code encoding is usually employed
in these proof constructions, and the inner verifier simply tests
whether a collection of tables (functions) are long codes satisfying
some constraints. Following this paradigm, constructing a $\PCP$
with certain parameters reduces to the problem of designing a long
code test with similar parameters.

One question of interest is the tradeoff between the soundness and
query complexity of a tester. If a tester queries the functions at
every single value, then trivially the verifier can determine all
the functions. One would like to construct a dictatorship test that
has the lowest possible soundness while making as few queries as possible.
One way to measure this tradeoff between the soundness $s$ and the
number of queries $q$ is \emph{amortized query complexity}, defined
as $\frac{q}{\log s^{-1}.}$ This investigation, initiated in~\cite{Tre98},
has since spurred a long sequence of works~\cite{SudTre98,SamTre00,HasWig01,EngHol05}.
All the testers from these works run many iterations of a single dictatorship
test by reusing queries from previous iterations. The techniques used
are Fourier analytic, and the best amortized query complexity from
this sequence of works has the form $1+O\left(\frac{1}{\sqrt{q}}\right)$.

The next breakthrough occurs when Samorodnitsky~\cite{Sam07} introduces
the notion of a \emph{relaxed} linearity test along with new ideas
from additive combinatorics. In property testing, the goal is to distinguish
objects that are very structured from those that are pseudorandom.
In the case of linearity/dictatorship testing, the structured objects
are the linear/dictator functions, and functions that are far from
being linear/dictator are interpreted as pseudorandom. The recent
paradigm in additive combinatorics is to find the right framework
of structure and pseudorandomness and analyze combinatorial objects
by dividing them into structured and pseudorandom components, see
e.g.~\cite{Tao07} for a survey. One success is the notion of Gowers
norm~\cite{Gow01}, which has been fruitful in attacking many problems
in additive combinatorics and computer science. In~\cite{Sam07},
the notion of pseudorandomness for linearity testing is relaxed; instead
of designating the functions that are far from being linear as pseudorandom,
the functions having small low degree Gowers norm are considered to
be pseudorandom. By doing so, an optimal tradeoff between soundness
and query complexity is obtained for the problem of relaxed linearity
testing. (Here the tradeoff is stronger than the tradeoff for the
traditional problem of linearity testing.)

In a similar fashion, in the $\PCP$ literature since~\cite{Has97},
the pseudorandom objects in dictatorship tests are not functions that
are far from being a dictator. The pseudorandom functions are typically
defined to be either functions that are far from all {}``juntas''
or functions whose {}``low-degree influences'' are $o(1)$. Both
considerations of a dictatorship test are sufficient to compose the
test in a $\PCP$ construction. In~\cite{SamTre06}, building on
the analysis of the relaxed linearity test in~\cite{Sam07}, Samorodnitsky
and Trevisan construct a dictatorship test (taking the view that functions
with arbitrary small {}``low-degree influences are pseudorandom)
with amortized query complexity $1+O\left(\frac{\log q}{q}\right)$.
Furthermore, the test is used as the inner verifier in a conditional
$\PCP$ construction (based on unique games~\cite{Kho02}) with the
same parameters. However, their dictatorship test suffers from an
inherent loss of perfect completeness. Ideally one would like testers
with one-sided errors. One, for aesthetic reasons, testers should
always accept valid inputs. Two, for some hardness of approximation
applications, in particular coloring problems (see e.g.~\cite{HasKho02}
or~\cite{DinMosReg06}), it is important to construct $\PCP$ systems
with one-sided errors.

In this paper, we prove the following theorem:

\begin{theorem}[main theorem] For every $q\geq3,$ there exists an
(adaptive) dictatorship test that makes $q$ queries, has completeness
$1,$ and soundness $\frac{O(q^{3})}{2^{q}};$ in particular it has
amortized query complexity $1+O\left(\frac{\log q}{q}\right)$. \end{theorem} 

Our tester is a variant of the one given in~\cite{SamTre06}. Our
tester is adaptive in the sense that it makes its queries in two stages.
It first makes roughly $\log q$ nonadaptive queries into the function.
Based on the values of these queries, the tester then selects the
rest of the query points nonadaptively. Our analysis is based on techniques
developed in~\cite{HasWig01,SamTre06,HasKho02,GLST98}.

\subsection{Future Direction}

Unfortunately, the adaptivity of our test is a drawback. The correspondence
between $\PCP$ constructions and hardness of approximation needs
the test to be fully nonadaptive. However, a more pressing issue is
that our hypergraph dictatorship test does not immediately imply a
new $\PCP$ characterization of $\NP$. The reason is that a dictatorship
test without {}``consistency checks'' is most easily composed with
the unique label cover defined in~\cite{Kho02} as the outer verifier
in a $\PCP$ reduction. As the conjectured $\NP$-hardness of the
unique label cover cannot have perfect completeness, the obvious approach
in combining our test with the unique games-based outer verifier does
not imply a new $\PCP$ result. However, there are variants of the
unique label cover (e.g., Khot's $d$ to $1$ Conjecture)~\cite{Kho02}
that do have conjectured perfect completeness, and these variants
are used to derive hardness of coloring problems in~\cite{DinMosReg06}.
We hope that our result combined with similar techniques used in~\cite{DinMosReg06}
may obtain a new conditional $\PCP$ construction and will motivate
more progress on constraint satisfaction problems with bounded projection
.

\subsection{Related Works}

The problem of linearity testing was first introduced in~\cite{BLR}.
The framework of property testing was formally set up in~\cite{RubSud96}.
The $\PCP$ Theorems were first proved in~\cite{AroSaf,ALMSS}; dictatorship
tests first appeared in the $\PCP$ context in~\cite{BelGolSud98},
and many dictatorship tests and variants appeared throughout the $\PCP$
literature. Dictatorship test was also considered as a standalone
property testing in~\cite{ParRonSam}. As mentioned, designing testers
and $\PCP$s focusing on amortized query complexity was first investigated
in~\cite{Tre98}, and a long sequence of works~\cite{SudTre98,SamTre00,HasWig01,EngHol05}
followed. The first tester/$\PCP$ system focusing on this tradeoff
while obtaining perfect completeness was achieved in~\cite{HasKho02}.

The orthogonal question of designing testers or $\PCP$s with as few
queries as possible was also considered. In a highly influential paper~\cite{Has97},
H\aa stad constructed a $\PCP$ system making only three queries.
Many variants also followed. In particular $\PCP$ systems with perfect
completeness making three queries were also achieved in~\cite{GLST98,KhoSak06}.
Similar to our approach, O'Donnell and Wu~\cite{ODonWu09a} designed
an optimal three bit dictatorship test with perfect completeness,
and later the same authors constructed a conditional $\PCP$ system~\cite{ODonWu09b}.

\section{Preliminaries}

We fix some notation and provide the necessary background in this
section. We let $[n]$ denote the set $\{1,2,\ldots,n\}$. For a vector
$v\in\cube^{n}$, we write $|v|=\sum_{i\in[n]}v_{i}$. We let $\wedge$
denote the boolean AND, where $a\wedge b=1$ iff $a=b=1$. For vectors
$v,w\in\cube^{n}$, we write $v\wedge w$ to denote the vector obtained
by applying AND to $v$ and $w$ component-wise. We abuse notation
and sometimes interpret a vector $v\in\cube^{n}$ as a subset $v\subseteq[n]$
where $i\in v$ iff $v_{i}=1.$ For a boolean function $f:\{0,1\}^{n}\rightarrow\cube$,
we make the convenient notational change from $\cube$ to $\pcube$
and write $f:\cube^{n}\rightarrow\pcube$.

\subsection{Fourier Analysis}

\begin{defn}[Fourier transform] For a real-valued function $f:\cube^{n}\rightarrow\mathbb{R}$,
we define its Fourier transform $\widehat{f}:\cube^{n}\rightarrow\mathbb{R}$
to be \[
\widehat{f}(\alpha)=\Exp_{x\in\cube^{n}}f(x)\chi_{\alpha}(x),\]
 where $\chi_{\alpha}(x)=(-1)^{\sum_{i\in[n]}\alpha_{i}x_{i}}$. We
say $\widehat{f}(\alpha)$ is the \emph{Fourier coefficient} of $f$
at $\alpha$, and the \emph{characters} of $\cube^{n}$ are the functions
$\{\chi_{\alpha}\}_{\alpha\in\cube^{n}}$. \end{defn} 

It is easy to see that for $\alpha,\beta\in\cube^{n}$, $\Exp\chi_{\alpha}\cdot\chi_{\beta}$
is 1 if $\alpha=\beta$ and $0$ otherwise. Since there are $2^{n}$
characters, they form an orthonormal basis for functions on $\cube^{n}$,
and we have the Fourier inversion formula\[
f(x)=\sum_{\alpha\in\cube^{n}}\widehat{f}(\alpha)\chi_{\alpha}(x)\]

and Parseval's Identity\[
\sum_{\alpha\in\cube^{n}}\widehat{f}(\alpha)^{2}=\Exp_{x}[f(x)^{2}].\]

\subsection{Influence of Variables}

For a boolean function $f:\cube^{n}\rightarrow\pcube$, the \emph{influence}
of the $i$-variable, $\infl_{i}(f)$, is defined to be $\Pr_{x\in\cube^{n}}[f(x)\neq f(x+e_{i})]$,
where $e_{i}$ is a vector in $\cube^{n}$ with $1$ on the $i$-th
coordinate $0$ everywhere else. This corresponds to our intuitive
notion of influence: how likely the outcome of $f$ changes when the
$i$-th variable on a random input is flipped. For the rest of this
paper, it will be convenient to work with the Fourier analytic definition
of $\infl_{i}(f)$ instead, and we leave it to the readers to verify
that the two definitions are equivalent when $f$ is a boolean function.

\begin{defn} Let $f:\cube^{n}\rightarrow\mathbb{R}$. We define the
influence of the $i$-th variable of $f$ to be \[
\infl_{i}(f)=\sum_{\alpha\in\cube^{n}:\ \alpha_{i}=1}\enspace\hat{f}(\alpha)^{2}.\]
 \end{defn}

We shall need the following technical lemma, which is Lemma $4$ from
~\cite{SamTre06}, and it gives an upper bound on the influence of
a product of functions.

\begin{lem}[\cite{SamTre06}]\label{lem: ST influence of product}Let
$f_{1},\ldots,f_{k}:\cube^{n}\rightarrow[-1,1]$ be a collection of
$k$ bounded real-valued functions, and define $f(x)=\prod_{i=1}^{k}f_{i}(x)$
to be the product of these $k$ functions. Then for each $i\in[n],$
\[
\infl_{i}(f)\leq k\cdot\sum_{j=1}^{k}\infl_{i}(f_{j}).\]

\end{lem}

When $\{f_{i}\}$ are boolean functions, it is easy to see that $\infl_{i}(f)\leq\sum_{j=1}^{k}\infl_{i}(f_{j})$
by the union bound.

We now define the notion of low-degree influence.

\begin{defn}Let $w$ be an integer between $0$ and $n.$ We define
the \emph{$w$-th degree influence of the $i$-th variable} of a function
$f:\cube^{n}\rightarrow\mathbb{R}$ to be \[
\infl_{i}^{\leq w}(f)=\sum_{\alpha\in\cube^{n}:\ \alpha_{i}=1,\ |\alpha|\leq w}\enspace\hat{f}(\alpha)^{2}.\]
 \end{defn}

While the definition of low-degree influence is standard in the literature,
we shall make a few remarks since this definition does not have a
clean combinatorial interpretation or an immediate justification.
Dictatorship tests (those based on influences) classify functions
in the NO instances to be those whose low-degree influences are $o(1)$
for two reasons. One is that large parity functions, which have many
variables with influence $1$ but no variables with low-degree influence,
must be rejected by the test. The second is that if $w$ is fixed,
then a bounded function has only a finite number of variables with
large $w$-th degree influence. This easy fact, though we won't need
it here, is often needed to lift a dictatorship test to a $\PCP$
construction. Both such considerations fail if we substitute the low-degree
influence requirement by just influence, thus the need for a thresholded
version of influence.

\subsection{Gowers norm}

In~\cite{Gow01}, Gowers uses analytic techniques to give a new proof
of Szeméredi's Theorem~\cite{Sze75} and in particular, initiates
the study of a new norm of a function as a measure of pseudorandomness.
Subsequently this norm is termed the \emph{Gowers uniformity norm}
and has been intensively studied and applied in additive combinatorics,
see e.g.~\cite{Tao07} for a survey. The use of the Gowers norm in
computer science is initiated in~\cite{Sam07,SamTre06}.

\begin{defn} Let $f:\cube^{n}\rightarrow\mathbb{R}$. We define the
\emph{$d$-th dimension Gowers uniformity norm} of $f$ to be \[
||f||_{U_{d}}=\left(\Exp_{x,\ x_{1},\ldots,x_{d}}\left[\thinspace\prod_{S\subseteq[d]}\thinspace f\left(x+\sum_{i\in S}x_{i}\right)\right]\right)^{1/2^{d}}.\]

For a collection of $2^{d}$ functions $f_{S}:\cube^{n}\rightarrow\mathbb{R},S\subset[d]$,
we define the \emph{$d$-th dimension Gowers inner product} of $\{f_{S}\}_{S\subseteq d}$
to be \[
\left\langle \{f_{S}\}_{S\subseteq[d]}\right\rangle _{U_{d}}=\Exp_{x,\ x_{1},\ldots,x_{d}}\thinspace\left[\prod_{S\subseteq[d]}\thinspace f_{S}\left(x+\sum_{i\in S}x_{i}\right)\right].\]
 \end{defn}

When $f$ is a boolean function, one can interpret the Gowers norm
as simply the expected number of {}``affine parallelepipeds'' of
dimension $d.$ While this expression may look cumbersome at first
glance, the use of the Gowers norm is in some sense to control expectations
over some other expressions. For instance, to count the number of
$d+1$-term progressions of the form $x,x+y,\ldots,x+d\cdot y$ in
a subset, one may be interested in approximating expressions of the
form $\Exp_{x,y}[f_{1}(x)f_{2}(x+y)\cdots f_{d}(x+d\cdot y)]$, where
$f_{1},\ldots,f_{d}$ are some bounded functions over some appropriate
domain. In fact, as shown by Gowers, these expectations are upper
bounded by the Gowers inner product of $f_{i}$, which is also upper
bounded by $\min_{i\in[d]}||f_{i}||_{U_{d}}^{2^{d}}$. Thus, in a
rough sense, questions regarding progressions are then reduced to
questions regarding the Gowers norms, which are more amenable to analytic
techniques.

The proof showing that $\Exp_{x,y}[f_{1}(x)f_{2}(x+y)\cdots f_{d}(x+d\cdot y)]$
is upper bounded by the minimum Gowers norm of all the functions $f_{i}$
is not difficult; it proceeds by repeated applications of the Cauchy-Schwarz
inequality and substitution of variables. Collectively, statements
saying that certain expressions are governed by the Gowers norm are
coined \emph{von-Neumann type theorems} in the literature.

For the analysis of hypergraph-based dictatorship test, we shall encounter
the following expression.

\begin{defn}Let $\{f_{S}\}_{S\subseteq[d]}$ be a collection of functions
where $f_{S}:\cube^{n}\rightarrow\mathbb{R}$. We define the \emph{$d$-th
dimension Gowers linear inner product} of $\{f_{S}\}$ to be \[
\left\langle \{f_{S}\}_{S\subseteq[d]}\right\rangle _{LU_{d}}=\Exp_{x_{1},\ldots,x_{d}}\thinspace\left[\prod_{S\subseteq[d]}\thinspace f_{S}\left(\sum_{i\in S}x_{i}\right)\right].\]
 \end{defn}

This definition is a variant of the Gowers inner product and is in
fact upper bounded by the square root of the Gowers inner product
as shown in~\cite{SamTre06}. Furthermore they showed that if a collection
of functions has large Gowers inner product, then two functions must
share an influential variable. Thus, one can infer the weaker statement
that large linear Gowers inner product implies two functions have
an influential variable.

For our purposes, we can encapsulate all the prior discussion into
the following statement, which is Lemma $16$ from~\cite{SamTre06}.
This is the only fact on the Gowers norm that we explicitly need.

\begin{lem}[\cite{SamTre06}]\label{lem:ST gowers inverse} Let $\{f_{S}\}_{S\subseteq[d]}$
be a collection of bounded functions of the form $f_{S}:\cube^{n}\rightarrow[-1,1].$
Suppose $\left\langle \{f_{S}\}_{S\subseteq[d]}\right\rangle _{LU_{d}}\geq\epsilon$
and $\Exp f_{[d]}=0$. Then there exists some variable $i$, some
subsets $S\neq T\subseteq[d]$ such that the influences of the $i$-th
variable in both $f_{S}$ and $f_{T}$ are at least $\frac{\eps^{4}}{2^{O(d)}}.$
\end{lem}

\section{Dictatorship Test}

\begin{defn}[dictatorship] For $i\in[n]$, the \emph{$i$-th dictator}
is the function $f(x)=(-1)^{x_{i}}$. \end{defn} In the $\PCP$ literature,
the $i$--th dictator is also known as the long code encoding of $i$,
$\left\langle (-1)^{x_{i}}\right\rangle _{x\in\cube^{n}}$, which
is simply the evaluation of the $i$-th dictator function at all points.

Now let us define a $t$-function dictatorship test. Suppose we are
given oracle access to a collection of boolean functions $f_{1},\ldots,f_{t}$.
We want to make as few queries as possible into these functions to
decide if all the functions are the same dictatorship, or no two functions
have some common structure. More precisely, we have the following
definition:

\begin{defn} We say that a test $T=T^{f_{1},\ldots,f_{t}}$ is a
\emph{$t$--function dictatorship test} with \emph{completeness} $c$
and \emph{soundness} $s$ if $T$ is given oracle access to a family
of $t$ functions $f_{1},\ldots,f_{t}:\cube^{n}\rightarrow\pcube$,
such that \end{defn} 
\begin{itemize}
\item if there exists some variable $i\in[n]$ such that for all $a\in[t],$
$f_{a}(x)=(-1)^{x_{i}}$, then $T$ accepts with probability at least
$c$, and 
\item for every $\epsilon>0$, there exist a positive constant $\tau>0$
and a fixed positive integer $w$ such that if $T$ accepts with probability
at least $s+\epsilon$, then there exist two functions $f_{a},f_{b}$
where $a,b\in[t],a\neq b$ and some variable $i\in[n]$ such that
$\infl_{i}^{\leq w}(f_{a}),\infl_{i}^{\leq w}(f_{b})\geq\tau$. 
\end{itemize}
A $q$-function dictatorship test making $q$ queries, with soundness
$\frac{q+1}{2^{q}}$ was proved in~\cite{SamTre06}, but the test
suffers from imperfect completeness. We obtain a $\left(q-O(\log q)\right)$--dictatorship
test that makes $q$ queries, has completeness $1$, soundness $\frac{O(q^{3})}{2^{q}}$,
and in particular has amortized query complexity $1+O\left(\frac{\log q}{q}\right)$,
the same as the test in~\cite{SamTre06}. By a simple change of variable,
we can more precisely state the following:

\begin{theorem}[main theorem restated] For infinitely many $t$,
there exists an adaptive $t$-function dictatorship test that makes
$t+\log(t+1)$ queries, has completeness $1$, and soundness $\frac{(t+1)^{2}}{2^{t}}$.
\end{theorem} 

Our test is adaptive and selects queries in two passes. During the
first pass, it picks an arbitrary subset of $\log(t+1)$ functions
out of the $t$ functions. For each function selected, our test picks
a random entry $y$ and queries the function at entry $y$. Then based
on the values of these $\log(t+1)$ queries, during the second pass,
the test selects $t$ positions nonadaptively, one from each function,
then queries all $t$ positions at once. The adaptivity is necessary
in our analysis, and it is unclear if one can prove an analogous result
with only one pass.

\subsection{Folding}

As introduced by Bellare, Goldreich, and Sudan~\cite{BelGolSud98},
we shall assume that the functions are {}``folded'' as only half
of the entries of a function are accessed. We require our dictatorship
test to make queries in a special manner. Suppose the test wants to
query $f$ at the point $x\in\cube^{n}$. If $x_{1}=1$, then the
test queries $f(x)$ as usual. If $x_{1}=0$, then the test queries
$f$ at the point $\vec{1}+x=(1,1+x_{2},\ldots,1+x_{n})$ and negates
the value it receives. It is instructive to note that folding ensures
$f(\vec{1}+x)=-f(x)$ and $\Exp f=0$.

\subsection{Basic Test}

For ease of exposition, we first consider the following simplistic
scenario. Suppose we have oracle access to just one boolean function.
Furthermore we ignore the tradeoff between soundness and query complexity.
We simply want a dictatorship test that has completeness $1$ and
soundness $\frac{1}{2}$. There are many such tests in the literature;
however, we need a suitable one which our hypergraph dictatorship
test can base on. Our basic test below is a close variant of the one
proposed by Guruswami, Lewin, Sudan, and Trevisan~\cite{GLST98}.

 \begin{breakbox} \noun{Basic Test $T$}: with oracle access
to $f$, 
\begin{enumerate}
\item Pick $x_{i},x_{j},y,z$ uniformly at random from $\cube^{n}$. 
\item Query $f(y)$. 
\item Let $v=\frac{1-f(y)}{2}.$ Accept iff \[
f(x_{i})f(x_{j})=f(x_{i}+x_{j}+(v\vec{1}+y)\wedge z).\]

\end{enumerate}
\end{breakbox}

\begin{lem} The test $T$ is a dictatorship test with completeness
$1.$ \end{lem} \begin{proof} Suppose $f$ is the $\ell$-th dictator,
i.e., $f(x)=(-1)^{x_{\ell}}.$ First note that \[
v+y_{\ell}=\frac{1-(-1)^{y_{\ell}}}{2}+y_{\ell},\]
 which evaluates to $0.$ Thus by linearity of $f$ \begin{eqnarray*}
f(x_{i}+x_{j}+(v\vec{1}+y)\wedge z) & = & f(x_{i})f(x_{j})f((v\vec{1}+y)\wedge z)\\
 & = & f(x_{i})f(x_{j})(-1)^{(v+y_{\ell})\wedge z_{\ell}}\\
 & = & f(x_{i})f(x_{j})\end{eqnarray*}
 and the test always accepts. \end{proof}

To analyze the soundness of the test $T$, we need to derive a Fourier
analytic expression for the acceptance probability of $T$.

\begin{prop} \label{prop:basic test fourier}Let $p$ be the acceptance
probability of $T$. Then \[
p=\frac{1}{2}+\frac{1}{2}\sum_{\alpha\in\cube^{n}}\widehat{f}(\alpha)^{3}\,2^{-|\alpha|}\left(1+\sum_{\beta\subseteq\alpha}\widehat{f}(\beta)\right).\]

\end{prop}

For sanity check, let us interpret the expression for $p.$ Suppose
$f=\chi_{\alpha}$ for some $\alpha\neq\vec{0}\in\cube^{n}$, i.e.,
$\widehat{f}(\alpha)=1$ and all other Fourier coefficients of $f$
are $0$. Then clearly $p=\frac{1}{2}+2^{-|\alpha|}$, which equals
$1$ whenever $f$ is a dictator function as we have just shown. If
$|\alpha|$ is large, then $T$ accepts with probability close to
$\frac{1}{2}.$ We shall first analyze the soundness and then derive
this analytic expression for $p.$ \\

\begin{lem} The test $T$ is a dictatorship test with soundness $\frac{1}{2}.$
\end{lem}

\begin{proof} Suppose the test $T$ passes with probability at least
$\frac{1}{2}+\eps,$ for some $\eps>0.$ By applying Proposition~\ref{prop:basic test fourier},
Cauchy-Schwarz, and Parseval's Identity, respectively, we
obtain\begin{eqnarray*}
\eps & \leq & \frac{1}{2}\sum_{\alpha\in\cube^{n}}\widehat{f}(\alpha)^{3}\,2^{-|\alpha|}\left(1+\sum_{\beta\subseteq\alpha}\widehat{f}(\beta)\right)\\
 & \leq & \frac{1}{2}\sum_{\alpha\in\cube^{n}}\widehat{f}(\alpha)^{3}\,2^{-|\alpha|}\left(1+\left(\sum_{\beta\subseteq\alpha}\widehat{f}(\beta)^{2}\right)^{\frac{1}{2}}\cdot2^{\frac{|\alpha|}{2}}\right)\\
 & \leq & \sum_{\alpha\in\cube^{n}}\widehat{f}(\alpha)^{3}\,2^{-\frac{|\alpha|}{2}}.\end{eqnarray*}

Pick the least positive integer $w$ such that $2^{-\frac{w}{2}}\leq\frac{\eps}{2}.$
Then by Parseval's again, \begin{eqnarray*}
\frac{\eps}{2} & \leq & \sum_{\alpha\in\cube^{n}:|\alpha|\leq w}\thinspace\widehat{f}(\alpha)^{3}\\
 & \leq & \max_{\alpha\in\cube^{n}:|\alpha|\leq w}\thinspace\left|\widehat{f}(\alpha)\right|.\end{eqnarray*}
 So there exists some $\beta\in\cube^{n},|\beta|\leq w$ such that
$\frac{\eps}{2}\leq\left|\hat{f}(\beta)\right|.$ With $f$ being
folded, $\beta\neq\vec{0}$. Thus, there exists an $i\in[n]$ such
that $\beta_{i}=1$ and \[
\frac{\eps^{2}}{4}\leq\widehat{f}(\beta)^{2}\leq\sum_{\alpha\in\cube^{n}:\alpha_{i}=1,|\alpha|\leq w}\thinspace\widehat{f}(\alpha)^{2}.\]
 \end{proof}

Now we give the straightforward Fourier analytic calculation for $p.$\\

\begin{proof}[Proof of Proposition \ref{prop:basic test fourier}]As
usual, we first arithmetize $p$. We write \begin{eqnarray*}
p & = & \Exp_{x_{i},x_{j},y,z}\left(\frac{1+f(y)}{2}\right)\left(\frac{1+\Acc(x_{i},x_{j},y,z)}{2}\right)+\\
 &  & \Exp_{x_{i},x_{j},y,z}\left(\frac{1-f(y)}{2}\right)\left(\frac{1+\Acc(x_{i},x_{j},\vec{1}+y,z)}{2}\right),\end{eqnarray*}
 where \[
\Acc(x_{i},x_{j},y,z)=f(x_{i})f(x_{j})f(x_{i}+x_{j}+(y\wedge z)).\]

Since $f$ is folded, $f(\vec{1}+y)=-f(y)$. As $y$ and $\vec{1}+y$
are both identically distributed in $\cube^{n}$, we have \[
p=2\Exp_{x_{i},x_{j},y,z}\left(\frac{1+f(y)}{2}\right)\left(\frac{1+\Acc(x_{i},x_{j},y,z)}{2}\right).\]

Since $\Exp f=0$, we can further simplify the above expression to
be \[
p=\frac{1}{2}+\frac{1}{2}\Exp_{x_{i},x_{j},y,z}\left[(1+f(y))\Acc(x_{i},x_{j},y,z)\right].\]

It suffices to expand out the terms $\Exp_{x_{i},x_{j},y,z}[\Acc(x_{i},x_{j},y,z)]$
and $\Exp_{x_{i},x_{j},y,z}[f(y)\Acc(x_{i},x_{j},y,z)]$. 

For the first term, it is not hard to show that \[
\Exp_{x_{i},x_{j},y,z}[\Acc(x_{i},x_{j},y,z)]=\sum_{\alpha\in\cube^{n}}\widehat{f}(\alpha)^{3}\enspace2^{-|\alpha|},\]
 by applying the Fourier inversion formula on $f$ and averaging over
$x_{i}$ and $x_{j}$ and then averaging over $y$ and $z$ over the
AND operator.

Now we compute the second term. Applying the Fourier inversion formula
to the last three occurrences of $f$ and averaging over $x_{i}$
and $x_{j}$, we obtain \[
\Exp_{x_{i},x_{j},y,z}[f(y)\Acc(x_{i},x_{j},y,z)]=\sum_{\alpha\in\cube^{n}}\widehat{f}(\alpha)^{3}\thinspace\Exp_{y,z}\left[f(y)\chi_{\alpha}(y\wedge z)\right].\]

It suffices to expand out $\Exp_{y,z}\left[f(y)\chi_{\alpha}(y\wedge z)\right]$.
By grouping the $z$'s according to their intersection with different
possible subsets $\beta$ of $\alpha$, we have \begin{eqnarray*}
\lefteqn{\Exp_{y,z}\left[f(y)\chi_{\alpha}(y\wedge z)\right]}\\
 & = & \sum_{\beta\subseteq\alpha}\thinspace\Pr_{z\in\cube^{n}}\left[z\cap\alpha=\beta\right]\thinspace\Exp_{y}\left[f(y)\prod_{i:\thinspace\alpha_{i}=1}(-1)^{y_{i}\wedge z_{i}}\right]\\
 & = & \sum_{\beta\subseteq\alpha}\thinspace2^{-|\alpha|}\thinspace\Exp_{y}\left[f(y)\prod_{i:\thinspace\beta_{i}=1}(-1)^{y_{i}}\right]\\
 & = & 2^{-|\alpha|}\thinspace\sum_{\beta\subseteq\alpha}\widehat{f}(\beta).\end{eqnarray*}
 Putting everything together, it is easy to see that we have the Fourier
analytic expression for $p$ as stated in the lemma. \end{proof}

\subsection{Hypergraph Dictatorship Test}

We prove the main theorem in this section. The basis of our hypergraph
dictatorship test will be very similar to the test in the previous
section. We remark that we did not choose to present the exact same
basic test for hopefully a clearer exposition.

We now address the tradeoff between query complexity and soundness.
If we simply repeat the basic test a number of iterations independently,
the error is reduced, but the query complexity increases. In other
words, the amortized query complexity does not change if we simply
run the basic test for many independent iterations. Following Trevisan~\cite{Tre98},
all the dictatorship tests that save query complexity do so by reusing
queries made in previous iterations of the basic test. To illustrate
this idea, suppose test $T$ queries $f$ at the points $x_{1}+h_{1}$,
$x_{2}+h_{2}$, $x_{1}+x_{2}+h_{1,2}$ to make a decision. For the
second iteration, we let $T$ query $f$ at the points $x_{3}+h_{3}$
and $x_{1}+x_{3}+h_{1,3}$ and reuse the value $f(x_{1}+h_{1})$ queried
during the first run of $T$. $T$ then uses the three values to make
a second decision. In total $T$ makes five queries to run two iterations.

We may think of the first run of $T$ as parametrized by the points
$x_{1}$ and $x_{2}$ and the second run of $T$ by $x_{1}$ and $x_{3}$.
In general, we may have $k$ points $x_{1},\ldots,x_{k}$ and a graph
on $[k]$ vertices, such that each edge $e$ of the graph corresponds
to an iteration of $T$ parametrized by the points $\{x_{i}\}_{i\in e}.$
We shall use a complete hypergraph on $k$ vertices to save on query
complexity, and we will argue that the soundness of the algorithm
decreases exponentially with respect to the number of iterations.

Formally, consider a hypergraph $H=([k],E)$. Let $\{f_{a}\}_{a\in[k]\cup E}$
be a collection of boolean functions of the form $f_{a}:\cube^{n}\rightarrow\pcube$.
We assume all the functions are folded, and so in particular, $\Exp f_{a}=0.$
Consider the following test:

\begin{breakbox} \noun{Hypergraph $H$-Test: }with oracle access
to $\{f_{a}\}_{a\in[k]\cup E}$, 
\begin{enumerate}
\item Pick $x_{1},\ldots,x_{k},y_{1},\ldots,y_{k}$, and $\{z_{a}\}_{a\in[k]\cup E}$
independently and uniformly at random from $\cube^{n}$. 
\item For each $i\in[k]$, query $f_{i}(y_{i})$. 
\item Let $v_{i}=\frac{1-f_{i}(y_{i})}{2}.$ \\
 Accept iff for every $e\in E$, \[
\prod_{i\in e}\left[f_{i}(x_{i}+(v_{i}\vec{1}+y_{i})\wedge z_{i})\right]=f_{e}\left(\sum_{i\in e}x_{i}+\left(\Sigma_{i\in e}(v_{i}\vec{1}+y_{i})\right)\wedge z_{e}\right).\]

\end{enumerate}
\end{breakbox}

We make a few remarks regarding the design of $H$-Test. The hypergraph
test by Samorodnitsky and Trevisan~\cite{SamTre06} accepts iff for
every $e\in E,$ $\prod_{i\in e}f_{i}(x_{i}+\eta_{i})$ equals $f_{e}(\sum_{i\in e}x_{i}+\eta_{e}),$
where the bits in each vector $\eta_{a}$ are chosen independently
to be $1$ with some small constant, say $0.01.$ The noise vectors
$\eta_{a}$ rule out the possibility that linear functions with large
support can be accepted. To obtain a test with perfect completeness,
we use ideas from~\cite{GLST98,ParRonSam,HasKho02} to simulate the
effect of the noise perturbation.

Note that for $y,z$ chosen uniformly at random from $\cube^{n},$
the vector $y\wedge z$ is a $\frac{1}{4}$--noisy vector. As observed
by Parnas, Ron, and Samorodnitsky~\cite{ParRonSam}, the test $f(y\wedge z)=f(y)\wedge f(z)$
distinguishes between dictators and linear functions with large support.
One can also combine linearity and dictatorship testing into a single
test of the form $f(x_{1}+x_{2}+y\wedge z)(f(y)\wedge f(z))=f(x_{1})f(x_{2})$
as Håstad and Khot demonstrated~\cite{HasKho02}. However, iterating
this test is too costly for us. In fact, Håstad and Khot also consider
an adaptive variant that reads $k^{2}+2k$ bits to obtain a soundness
of $2^{-k^{2}}$, the same parameters as in~\cite{SamTre00}, while
achieving perfect completeness as well. Without adaptivity, the test
in~\cite{HasKho02} reads $k^{2}+4k$ bits. While both the nonadaptive
and adaptive tests in~\cite{HasKho02} have the same amortized query
complexity, extending the nonadaptive test by Håstad and Khot to the
hypergraph setting does not work for us. So to achieve the same amortized
query complexity as the hypergraph test in~\cite{SamTre06}, we also
exploit adaptivity in our test.

\begin{theorem}[main theorem restated] For infinitely many $t$,
there exists an adaptive $t$-function dictatorship test with $t+\log(t+1)$
queries, completeness $1$, and soundness $\frac{(t+1)^{2}}{2^{t}}$.
\end{theorem}

\begin{proof} Take a complete hypergraph on $k$ vertices, where
$k=\log(t+1).$ The statement follows by applying Lemmas \ref{lem:H-Test completeness}
and \ref{lem:H-Test soundness}. \end{proof}

\begin{lem} \label{lem:H-Test completeness}The $H$-Test is a $(k+|E|)$-function
dictatorship test that makes $|E|+2k$ queries and has completeness
$1$. \end{lem}

\begin{proof} The test makes $k$ queries $f_{i}(y_{i})$ in the
first pass, and based on the answers to these $k$ queries, the test
then makes one query into each function $f_{a},$ for each $a\in[k]\cup E.$
So the total number of queries is $|E|+2k.$

Now suppose all the functions are the $\ell$-th dictator for some
$\ell\in[n]$, i.e., for all $a\in[k]\cup E,$ $f_{a}=f,$ where $f(x)=(-1)^{x_{\ell}}$.
Note that for each $i\in[k],$ \[
v_{i}+y_{i}(\ell)=\frac{1-(-1)^{y_{i}(\ell)}}{2}+y_{i}(\ell),\]
 which evaluates to $0$. Thus for each $e\in E,$ \begin{eqnarray*}
\prod_{i\in e}f_{i}(x_{i}+(v_{i}\vec{1}+y_{i})\wedge z_{i}) & = & f\left(\sum_{i\in e}x_{i}\right)\cdot\prod_{i\in e}f((v_{i}\vec{1}+y_{i})\wedge z_{i})\\
 & = & f\left(\sum_{i\in e}x_{i}\right)\cdot\prod_{i\in e}(-1)^{(v_{i}+y_{i}(\ell))\wedge z_{i}(\ell)}\\
 & = & f\left(\sum_{i\in e}x_{i}\right),\end{eqnarray*}

and similarly, \[
f_{e}\left(\sum_{i\in e}x_{i}+\left(\Sigma_{i\in e}(v_{i}\vec{1}+y_{i})\right)\wedge z_{e}\right)=f\left(\sum_{i\in e}x_{i}\right).\]
 Hence the test always accepts. \end{proof} 

\begin{lem} \label{lem:H-Test soundness}The $H$-Test has soundness
$2^{k-|E|}.$ \end{lem} 

Before proving Lemma \ref{lem:H-Test soundness} we first prove a
proposition relating the Fourier transform of a function perturbed
by noise to the function's Fourier transform itself.

\begin{prop} \label{pro:noisy function}Let $f:\cube^{n}\rightarrow\pcube.$
Define $g:\cube^{2n}\rightarrow[-1,1]$ to be \[
g(x;y)=\Exp_{z\in\cube^{n}}f(c'+x+(c+y)\wedge z),\]
 where $c,c'$ are some fixed vectors in $\cube^{n}.$ Then \[
\widehat{g}(\alpha;\beta)^{2}=\widehat{f}(\alpha)^{2}\thinspace1_{\{\beta\subseteq\alpha\}}4^{-|\alpha|}.\]

\end{prop}

\begin{proof} This is a straightforward Fourier analytic calculation.
By definition, \[
\widehat{g}(\alpha;\beta)^{2}=\left(\Exp_{x,y,z\in\cube^{n}}f(c'+x+(c+y)\wedge z)\chi_{\alpha}(x)\chi_{\beta}(y)\right)^{2}.\]
 By averaging over $x$ it is easy to see that \[
\widehat{g}(\alpha;\beta)^{2}=\widehat{f}(\alpha)^{2}\left(\Exp_{y,z\in\cube^{n}}\chi_{\alpha}((c+y)\wedge z)\chi_{\beta}(y)\right)^{2}.\]

Since the bits of $y$ are chosen independently and uniformly at random,
if $\beta\backslash\alpha$ is nonempty, the above expression is zero.
So we can write \[
\widehat{g}(\alpha;\beta)^{2}=\widehat{f}(\alpha)^{2}\thinspace1_{\{\beta\subseteq\alpha\}}\left(\prod_{i\in\alpha\backslash\beta}\thinspace\Exp_{y_{i},z_{i}}(-1)^{(c_{i}+y_{i})\wedge z_{i}}\cdot\prod_{i\in\beta}\Exp_{y_{i},z_{i}}(-1)^{(c_{i}+y_{i})\wedge z_{i}+y_{i}}\right)^{2}.\]

It is easy to see that the term $\Exp_{y_{i},z_{i}}(-1)^{(c_{i}+y_{i})\wedge z_{i}}$
evaluates to $\frac{1}{2}$ and the term $\Exp_{y_{i},z_{i}}(-1)^{(c_{i}+y_{i})\wedge z_{i}+y_{i}}$
evaluates to $(-1)^{c_{i}}\frac{1}{2}.$ Thus \[
\widehat{g}(\alpha;\beta)^{2}=\widehat{f}(\alpha)^{2}\enspace1_{\{\beta\subseteq\alpha\}}\thinspace4^{-|\alpha|}\]
 as claimed. \end{proof}

Now we prove Lemma \ref{lem:H-Test soundness}.\\

\begin{proof}[Proof of Lemma \ref{lem:H-Test soundness}] Let $p$
be the acceptance probability of $H$-test. Suppose that $2^{k-|E|}+\epsilon\leq p$.
We want to show that there are two functions $f_{a}$ and $f_{b}$
such that for some $i\in[n]$, some fixed positive integer $w,$ some
constant $\eps'>0$, it is the case that $\infl_{i}^{\leq w}(f_{a}),\infl_{i}^{\leq w}(f_{b})\geq\eps'.$
As usual we first arithmetize $p$. We write \[
p=\sum_{v\in\{0,1\}^{k}}\thinspace\Exp_{\{x_{i}\},\{y_{i}\},\{z_{a}\}}\thinspace\prod_{i\in[k]}\frac{1+(-1)^{v_{i}}f_{i}(y_{i})}{2}\thinspace\prod_{e\in E}\frac{1+\Acc(\{x_{i},y_{i},v_{i},z_{i}\}_{i\in e},z_{e})}{2},\]

where \begin{eqnarray*}
\Acc(\{x_{i},y_{i},v_{i},z_{i}\}_{i\in e},z_{e}) & = & \prod_{i\in e}\left[f_{i}(x_{i}+(v_{i}\vec{1}+y_{i})\wedge z_{i})\right]\\
 &  & \cdot\thinspace f_{e}\left(\sum_{i\in e}x_{i}+\left(\Sigma_{i\in e}(v_{i}\vec{1}+y_{i})\right)\wedge z_{e}\right).\end{eqnarray*}

For each $i\in[k],$ $f_{i}$ is folded, so $(-1)^{v_{i}}f_{i}(y_{i})=f_{i}(v_{i}\vec{1}+y_{i}).$
Since the vectors $\{y_{i}\}_{i\in[k]}$ are uniformly and independently
chosen from $\cube^{n},$ for a fixed $v\in\cube^{k},$ the vectors
$\{v_{i}\vec{1}+y_{i}\}_{i\in[k]}$ are also uniformly and independently
chosen from $\cube^{n}.$ So we can simplify the expression for $p$
and write \[
p=\Exp_{\{x_{i}\},\{y_{i}\},\{z_{a}\}}\left[\thinspace\prod_{i\in[k]}\left(1+f_{i}(y_{i})\right)\thinspace\prod_{e\in E}\frac{1+(\Acc\{x_{i},y_{i},\vec{0},z_{i}\}_{i\in e},z_{e})}{2}\right].\]
 Instead of writing $\Acc(\{x_{i},y_{i},\vec{0},z_{i}\}_{i\in e},z_{e}),$
for convenience we shall write $\Acc(e)$ to be a notational shorthand.
Observe that since $1+f_{i}(y_{i})$ is either $0$ or $2$, we may
write \[
p\leq2^{k}\Exp_{\{x_{i}\},\{y_{i}\},\{z_{a}\}}\left[\thinspace\prod_{e\in E}\frac{1+\Acc(e)}{2}\right].\]

Note that the product of sums $\prod_{e\in E}\frac{1+\Acc(e)}{2}$
expands into a sum of products of the form \[
2^{-|E|}\left(1+\sum_{\emptyset\neq E'\subseteq E}\enspace\prod_{e\in E'}\Acc(e)\right),\]
 so we have \[
\frac{\eps}{2^{k}}\leq\Exp_{\{x_{i}\},\{y_{i}\},\{z_{a}\}}\left[\thinspace2^{-|E|}\sum_{\emptyset\neq E'\subseteq E}\enspace\prod_{e\in E'}\Acc(e)\right].\]

By averaging, there must exist some nonempty subset $E'\subseteq E$
such that \[
\frac{\eps}{2^{k}}\leq\Exp_{\{x_{i}\},\{y_{i}\},\{z_{a}\}}\left[\thinspace\prod_{e\in E'}\Acc(e)\right].\]

Let $\Odd$ consists of the vertices in $[k]$ with odd degree in
$E'.$ Expanding out the definition of $\Acc(e),$ we can conclude

\[
\frac{\eps}{2^{k}}\leq\Exp_{\{x_{i}\},\{y_{i}\},\{z_{a}\}}\thinspace\left[\prod_{i\in\Odd}f_{i}(x_{i}+y_{i}\wedge z_{i})\cdot\prod_{e\in E'}f_{e}\left(\sum_{i\in e}x_{i}+\left(\sum_{i\in e}y_{i}\right)\wedge z_{e}\right)\right].\]
 \\

We now define a family of functions that represent the {}``noisy
versions'' of $f_{a}.$ For $a\in[k]\cup E,$ define $g_{a}':\cube^{2n}\rightarrow[-1,1]$
to be \[
g_{a}'(x;y)=\Exp_{z\in\cube^{n}}f_{a}(x+y\wedge z).\]

Thus we have \[
\frac{\eps}{2^{k}}\leq\Exp_{\{x_{i}\},\{y_{i}\}}\thinspace\left[\prod_{i\in\Odd}g'_{i}(x_{i};y_{i})\cdot\prod_{e\in E'}g'_{e}\left(\sum_{i\in e}x_{i};\sum_{i\in e}y_{i}\right)\right].\]
 Following the approach in~\cite{HasWig01,SamTre06}, we are going
to reduce the analysis of the iterated test to one hyperedge. Let
$d$ be the maximum size of an edge in $E',$ and without loss of
generality, let $(1,2,\ldots,d)$ be a maximal edge in $E'.$ Now,
fix the values of $x_{d+1},\ldots,x_{k}$ and $y_{d+1},\ldots,y_{k}$
so that the following inequality holds: \begin{equation}
\frac{\eps}{2^{k}}\leq\Exp_{x_{1},y_{1},\ldots,x_{d},y_{d}}\left[\prod_{i\in\Odd}g'_{i}(x_{i};y_{i})\cdot\prod_{e\in E'}g'_{e}\left(\sum_{i\in e}x_{i};\sum_{i\in e}y_{i}\right)\right].\label{eq:HW averaging inequality 1}\end{equation}

We group the edges in $E'$ based on their intersection with $(1,\ldots,d).$
We rewrite Inequality \ref{eq:HW averaging inequality 1} as \begin{equation}
\frac{\eps}{2^{k}}\leq\Exp_{(x_{1},y_{1}),\ldots,(x_{d},y_{d})\in\cube^{2n}}\left[\prod_{S\subseteq[d]}\enspace\prod_{a\in\Odd\cup E':a\cap[d]=S}\enspace g_{a}\left(\sum_{i\in S}x_{i};\sum_{i\in S}y_{i}\right)\right],\label{eq:HW averaging inequality 2}\end{equation}

where for each $a\in[k]\cup E,$ $g_{a}(x;y)=g'_{a}(c'_{a}+x;c_{a}+y)$,
with $c'_{a}=\sum_{i\in a\backslash[d]}\thinspace x_{i}$ and $c_{a}=\sum_{i\in a\backslash[d]}y_{i}$
fixed vectors in $\cube^{n}.$

By grouping the edges based on their intersection with $[d],$ we
can rewrite Inequality \ref{eq:HW averaging inequality 2} as \begin{eqnarray*}
\frac{\eps}{2^{k}} & \leq & \Exp_{(x_{1},y_{1}),\ldots,(x_{d},y_{d})\in\cube^{2n}}\left[\prod_{S\subseteq[d]}G_{S}\left(\sum_{i\in S}(x_{i};y_{i})\right)\right]\\
 & = & \left\langle \{G_{S}\}_{S\subseteq[d]}\right\rangle _{LU_{d}},\end{eqnarray*}
 where $G_{S}$ is simply the product of all the functions $g_{a}$
such that $a\in\Odd\cup E'$ and $a\cap[d]=S.$\\

Since $(1,\ldots,d)$ is maximal, all the other edges in $E'$ do
not contain $(1,\ldots,d)$ as a subset. Thus $G_{[d]}=g_{[d]}$ and
$\Exp G_{[d]}=0.$ By Lemma \ref{lem:ST gowers inverse}, the linear
Gowers inner product of a family of functions $\{G_{S}\}$ being positive
implies that two functions from the family must share a variable with
positive influence. More precisely, there exist $S\neq T\subseteq[d],$
$i\in[2n],$ $\tau>0,$ such that $\infl_{i}(G_{S}),\infl_{i}(G_{T})\geq\tau,$
where $\tau=\frac{\eps^{4}}{2^{O(d)}}.$

Note that $G_{\emptyset}$ is the product of all the functions $g'_{a}$
that are indexed by vertices or edges outside of $[d].$ So $G_{\emptyset}$
is a constant function, and all of its variables clearly have influence
$0.$ Thus neither $S$ nor $T$ is empty. Since $G_{S}$ and $G_{T}$
are products of at most $2^{k}$ functions, by Lemma \ref{lem: ST influence of product}
there must exist some $a\neq b\in[d]\cup E'$ such that $\infl_{i}(g_{a}),\infl_{i}(g_{b})\geq\frac{\tau}{2^{2k}}.$
Recall that we have defined $g_{a}(x;y)$ to be $\Exp_{z}f_{a}(c'_{a}+x+(c_{a}+y)\wedge z).$
Thus we can apply Proposition \ref{pro:noisy function} to obtain
\begin{eqnarray*}
\infl_{i}(g_{a}) & = & \sum_{(\alpha,\beta)\in\cube^{2n};i\in(\alpha,\beta)}\enspace\widehat{g}_{a}(\alpha;\beta)^{2}\\
 & = & \sum_{\alpha\in\cube^{n};i\in\alpha}\enspace\sum_{\beta\subseteq\alpha}\thinspace\hat{f_{a}}(\alpha)^{2}\thinspace4^{-|\alpha|}\\
 & = & \sum_{\alpha\in\cube^{n};i\in\alpha}\enspace\hat{f_{a}}(\alpha)^{2}\thinspace2^{-|\alpha|}.\end{eqnarray*}

Let $w$ be the least positive integer such that $2^{-w}\leq\frac{\tau}{2^{2k+1}}.$
Then it is easy to see that $\infl_{i}^{\leq w}(f_{a})\geq\frac{\tau}{2^{2k+1}}$.
Similarly, $\infl_{i}^{\leq w}(f_{b})\geq\frac{\tau}{2^{2k+1}}$ as
well. Hence this completes the proof. \end{proof}

\section{Acknowledgments}

I am grateful to Alex Samorodnitsky for many useful discussions and
his help with the Gowers norm. I also thank Madhu Sudan for his advice
and support and Swastik Kopparty for an encouraging discussion during
the initial stage of this research.

\bibliographystyle{plain} \bibliographystyle{plain} \bibliographystyle{plain}
\bibliography{main}

\begin{thebibliography}{10}

\bibitem{ALMSS}
Sanjeev Arora, Carsten Lund, Rajeev Motwani, Madhu Sudan, and Mario Szegedy.
\newblock Proof verification and the hardness of approximation problems.
\newblock {\em J. ACM}, 45(3):501--555, 1998.

\bibitem{AroSaf}
Sanjeev Arora and Shmuel Safra.
\newblock Probabilistic checking of proofs: a new characterization of {NP}.
\newblock {\em J. ACM}, 45(1):70--122, 1998.

\bibitem{BelGolSud98}
Mihir Bellare, Oded Goldreich, and Madhu Sudan.
\newblock Free bits, {PCP}s, and nonapproximability---towards tight results.
\newblock {\em SIAM Journal on Computing}, 27(3):804--915, 1998.

\bibitem{BLR}
Manuel Blum, Michael Luby, and Ronitt Rubinfeld.
\newblock Self-testing/correcting with applications to numerical problems.
\newblock {\em Journal of Computer and System Sciences}, 47(3):549--595, 1993.

\bibitem{DinMosReg06}
Irit Dinur, Elchanan Mossel, and Oded Regev.
\newblock Conditional hardness for approximate coloring.
\newblock In {\em STOC '06: Proceedings of the thirty-eighth annual ACM
  symposium on Theory of computing}, pages 344--353, New York, NY, USA, 2006.
  ACM.

\bibitem{EngHol05}
Lars Engebretsen and Jonas Holmerin.
\newblock More efficient queries in {PCP}s for {NP} and improved approximation
  hardness of maximum csp.
\newblock In {\em STACS}, pages 194--205, 2005.

\bibitem{Gow01}
W.~T. Gowers.
\newblock A new proof of {S}zemer{\'e}di's theorem.
\newblock {\em Geom. Funct. Anal.}, 11(3):465--588, 2001.

\bibitem{GLST98}
Venkatesan Guruswami, Daniel Lewin, Madhu Sudan, and Luca Trevisan.
\newblock A tight characterization of {NP} with 3 query {PCP}s.
\newblock In {\em FOCS '98: Proceedings of the 39th Annual Symposium on
  Foundations of Computer Science}, page~8, Washington, DC, USA, 1998. IEEE
  Computer Society.

\bibitem{Has97}
Johan H{\aa}stad.
\newblock Some optimal inapproximability results.
\newblock {\em J. ACM}, 48(4):798--859, 2001.

\bibitem{HasKho02}
Johan H{\aa}stad and Subhash Khot.
\newblock Query efficient {PCP}s with perfect completeness.
\newblock {\em Theory of Computing}, 1(7):119--148, 2005.

\bibitem{HasWig01}
Johan H{\aa}stad and Avi Wigderson.
\newblock Simple analysis of graph tests for linearity and {PCP}.
\newblock {\em Random Struct. Algorithms}, 22(2):139--160, 2003.

\bibitem{Kho02}
Subhash Khot.
\newblock On the power of unique 2-prover 1-round games.
\newblock In {\em STOC '02: Proceedings of the thiry-fourth annual ACM
  symposium on Theory of computing}, pages 767--775, New York, NY, USA, 2002.
  ACM.

\bibitem{KhoSak06}
Subhash Khot and Rishi Saket.
\newblock A 3-query non-adaptive {PCP} with perfect completeness.
\newblock In {\em CCC '06: Proceedings of the 21st Annual IEEE Conference on
  Computational Complexity}, pages 159--169, Washington, DC, USA, 2006. IEEE
  Computer Society.

\bibitem{ODonWu09a}
Ryan O'Donnell and Yi~Wu.
\newblock 3-bit dictator testing: 1 vs. 5/8.
\newblock In {\em SODA '09: Proceedings of the Nineteenth Annual ACM -SIAM
  Symposium on Discrete Algorithms}, pages 365--373, Philadelphia, PA, USA,
  2009. Society for Industrial and Applied Mathematics.

\bibitem{ODonWu09b}
Ryan O'Donnell and Yi~Wu.
\newblock Conditional hardness for satisfiable-3csps.
\newblock In {\em STOC '09: Proceedings of the forty-first annual ACM symposium
  on Theory of computing}, page To appear, 2009.

\bibitem{ParRonSam}
Michal Parnas, Dana Ron, and Alex Samorodnitsky.
\newblock Testing basic boolean formulae.
\newblock {\em SIAM Journal on Discrete Mathematics}, 16(1):20--46, 2002.

\bibitem{Raz98}
Ran Raz.
\newblock A parallel repetition theorem.
\newblock {\em SIAM J. Comput.}, 27(3):763--803, 1998.

\bibitem{RubSud96}
Ronitt Rubinfeld and Madhu Sudan.
\newblock Robust characterizations of polynomials with applications to program
  testing.
\newblock {\em SIAM Journal on Computing}, 25(2):252--271, 1996.

\bibitem{Sam07}
Alex Samorodnitsky.
\newblock Low-degree tests at large distances.
\newblock In {\em STOC '07: Proceedings of the thirty-ninth annual ACM
  symposium on Theory of computing}, pages 506--515, New York, NY, USA, 2007.
  ACM.

\bibitem{SamTre00}
Alex Samorodnitsky and Luca Trevisan.
\newblock A {PCP} characterization of {NP} with optimal amortized query
  complexity.
\newblock In {\em STOC '00: Proceedings of the thirty-second annual ACM
  symposium on Theory of computing}, pages 191--199, New York, NY, USA, 2000.
  ACM.

\bibitem{SamTre06}
Alex Samorodnitsky and Luca Trevisan.
\newblock Gowers uniformity, influence of variables, and {PCP}s.
\newblock In {\em STOC '06: Proceedings of the thirty-eighth annual ACM
  symposium on Theory of computing}, pages 11--20, New York, NY, USA, 2006.
  ACM.

\bibitem{SudTre98}
Madhu Sudan and Luca Trevisan.
\newblock Probabilistically checkable proofs with low amortized query
  complexity.
\newblock In {\em FOCS '98: Proceedings of the 39th Annual Symposium on
  Foundations of Computer Science}, page~18, Washington, DC, USA, 1998. IEEE
  Computer Society.

\bibitem{Sze75}
Endre Szemer{\'e}di.
\newblock On sets of integers containing no {$k$} elements in arithmetic
  progression.
\newblock {\em Acta Arith.}, 27:199--245, 1975.

\bibitem{Tao07}
Terence Tao.
\newblock Structure and randomness in combinatorics.
\newblock In {\em FOCS '07: Proceedings of the forty-eighth annual ACM
  symposium on Foundations of computer science}, pages 3--15, New York, NY,
  USA, 2007. ACM.

\bibitem{Tre98}
Luca Trevisan.
\newblock Recycling queries in {PCP}s and in linearity tests (extended
  abstract).
\newblock In {\em STOC '98: Proceedings of the thirtieth annual ACM symposium
  on Theory of computing}, pages 299--308, New York, NY, USA, 1998. ACM.

\end{thebibliography}

\end{document}